\documentclass[11pt,a4paper]{amsart}
\usepackage{amsfonts,amssymb}
\usepackage{amsthm,amsmath,amstext,xfrac}
\usepackage{enumitem}
\usepackage{tikz}
\usetikzlibrary{matrix,arrows,calc}
\usepackage{graphicx,mathrsfs,ifpdf}
\usepackage{todonotes}

\newtheorem{theorem}{Theorem}[section]
\newtheorem{lemma}[theorem]{Lemma}

\newtheorem{corollary}[theorem]{Corollary}

\theoremstyle{definition}

\numberwithin{equation}{section}
\numberwithin{figure}{section}

\DeclareMathOperator{\Aut}{Aut}

\DeclareMathOperator{\Mt}{Mt}
\DeclareMathOperator{\inter}{int}
\DeclareMathOperator{\odd}{odd}
\DeclareMathOperator{\NN}{\mathbb N}
\DeclareMathOperator{\RR}{\mathbb R}
\DeclareMathOperator{\ZZ}{\mathbb Z}
\newcommand{\area}[1]{\mathrm{area}(#1)}

\newenvironment{tikzgraph}
  {\begin{tikzpicture}
      [vertex/.style={circle, draw=black, fill, inner sep=0mm,%
      minimum size=3pt},edge/.style={semithick},
       subdivision/.style={circle, draw=black, fill=white, inner sep=0mm,%
       minimum size=3pt},edge/.style={semithick}]\begin{scope}}
  {\end{scope}\end{tikzpicture}}

\begin{document}
\title[Odd cycles in cubic plane graphs]{Packing and covering odd cycles in cubic plane graphs with small faces}
\author{Diego Nicodemos}
\thanks{The first author was partially supported by CAPES and CNPq.}
\address{Col\'egio Pedro II, COPPE/Sistemas, Universidade Federal do Rio
  de Janeiro, Brazil}
\email{nicodemos@cos.ufrj.br}
\author{Mat\v ej Stehl\'ik}
\thanks{The second author was partially supported by ANR project Stint (ANR-13-BS02-0007), ANR
project GATO (ANR-16-CE40-0009-01), and by LabEx PERSYVAL-Lab (ANR-11-LABX-0025).}
\address{Laboratoire G-SCOP, Univ.\ Grenoble Alpes, France}
\email{matej.stehlik@grenoble-inp.fr}
\date{}

\begin{abstract}
  We show that any $3$-connected cubic plane graph on $n$ vertices,
  with all faces of size at most $6$, can be made bipartite by deleting
  no more than $\sqrt{(p+3t)n/5}$ edges, where $p$ and $t$ are the numbers
  of pentagonal and triangular faces, respectively. In particular, any such
  graph can be made bipartite by deleting at most $\sqrt{12n/5}$ edges.
  This bound is tight, and we characterise the extremal graphs. We deduce
  tight lower bounds on the size of a maximum cut and a maximum independent
  set for this class of graphs. This extends and sharpens the results of
  Faria, Klein and Stehl\'ik [SIAM J.\ Discrete Math.\ 26 (2012) 1458--1469].
\end{abstract}

\maketitle

\section{Introduction}

A set of edges intersecting every odd cycle in a graph is known as an
\emph{odd cycle (edge) transversal}, or \emph{odd cycle cover}, and the
minimum size of such a set is denoted by $\tau_{\odd}$. A set of edge-disjoint
odd cycles in a graph is called a \emph{packing of odd cycles}, and the maximum
size of such a family is denoted by $\nu_{\odd}$. Clearly, $\tau_{\odd} \geq \nu_{\odd}$.
Dejter and Neumann-Lara~\cite{DNL88} and independently
Reed~\cite{Ree99} showed that in general, $\tau_{\odd}$ cannot be bounded
by any function of $\nu_{\odd}$, i.e., they do not satisfy the
Erd\H os--P\'osa property. However, for planar graphs, Kr\'al\!'\ and
Voss~\cite{KraVos04} proved the (tight) bound $\tau_{\odd}\leq 2\nu_{\odd}$.

In this paper we focus on packing and covering of odd cycles in $3$-connected
cubic plane graphs with all faces of size at most $6$. Such graphs---and their
dual triangulations---are a very natural class to consider, as they correspond
to surfaces of genus $0$ of non-negative curvature (see e.g.~\cite{Thu98}).

A much-studied subclass of cubic plane graphs with all faces of size at most $6$
is the class of \emph{fullerene graphs}, which only have faces of size $5$ and $6$.
Faria, Klein and Stehl\'ik~\cite{FKS12} showed that any fullerene
graph on $n$ vertices has an odd cycle transversal with no more than
$\sqrt{12n/5}$ edges, and characterised the extremal graphs. Our main result
is the following extension and sharpening of their result to all $3$-connected cubic
plane graphs with all faces of size at most $6$.

\begin{theorem}
\label{thm:OCT}
  Let $G$ be a $3$-connected cubic plane graph on $n$ vertices with all faces of size at most
  $6$, with $p$ pentagonal and $t$ triangular faces. Then
  \[
    \tau_{\odd}(G) \leq \sqrt{(p+3t)n/5}.
  \]
  In particular, $\tau_{\odd}(G) \leq \sqrt{12n/5}$ always holds, with equality
  if and only if all faces have size $5$ and $6$, $n=60k^2$ for some
  $k \in \NN$, and $\Aut(G) \cong I_h$.
\end{theorem}

If $G$ is a fullerene graph, then $t=0$ and Euler's formula implies that $p=12$, so
Theorem~\ref{thm:OCT} does indeed generalise the result of Faria, Klein and Stehl\'ik~\cite{FKS12}.
We also remark that the smallest $3$-connected cubic plane graph with all faces of size at most $6$
achieving the bound $\tau_{\odd}(G) = \sqrt{12n/5}$ in Theorem~\ref{thm:OCT} is the ubiquitous
\emph{buckminsterfullerene graph} (on 60 vertices).

The rest of the paper is organised as follows. In Section~\ref{sec:preliminaries},
we introduce the basic notation and terminology, as well as the key concepts from
combinatorial optimisation and topology. In Section~\ref{sec:patches}, we introduce
the notions of patches and moats, and prove bounds on the area of moats. Then, in
Section~\ref{sec:T-cut}, we use these bounds to
prove an upper bound on the maximum size of a packing of $T$-cuts in triangulations
of the sphere with maximum degree at most $6$. Using a theorem of Seymour~\cite{Sey81},
we deduce, in Section~\ref{sec:T-join}, an upper bound on the minimum
size of a $T$-join in triangulations of the sphere with maximum degree at most $6$,
and then dualise to complete the proof of Theorem~\ref{thm:OCT}. In
Section~\ref{sec:max-cut}, we deduce lower bounds on the size of a maximum cut and
a maximum independent set in $3$-connected cubic plane graphs with no faces of size
more than $6$. Finally, in Section~\ref{sec:conclusion}, we show why the condition
on the face size cannot be relaxed, and briefly discuss the special case when the graph
contains no pentagonal faces.

\section{Preliminaries}
\label{sec:preliminaries}

Most of our graph-theoretic terminology is standard and follows~\cite{BonMur08}.
All graphs are finite and simple, i.e., have no loops and parallel edges.
The \emph{degree} of a vertex $u$ in a graph $G$ is denoted by $d_G(u)$.
If all vertices in $G$ have degree $3$, then $G$ is a \emph{cubic graph}.
The set of edges in $G$ with exactly one end vertex in $X$ is denoted by $\delta_G(X)$.
A set $C$ of edges is a \emph{cut} of $G$ if $C=\delta_G(X)$, for some $X \subseteq V(G)$.
When there is no risk of ambiguity, we may omit the subscripts in the above notation.

The set of all automorphisms of a graph $G$ forms a group, known as the \emph{automorphism group
$\Aut(G)$}. The \emph{full icosahedral group} $I_h \cong A_5 \times C_2$ is the group of
all symmetries (including reflections) of the regular icosahedron. The \emph{full tetrahedral
group} $T_d \cong S_4$ is the group of all symmetries (including reflections) of the regular
tetrahedron.

A \emph{polygonal surface} $K$ is a simply connected $2$-manifold, possibly with a boundary,
which is obtained from a finite collection of disjoint simple polygons in $\RR^2$ by
identifying them along edges of equal length. We denote by $|K|$ the union of all polygons in
$K$, and remark that $|K|$ is a surface.

Based on this construction, $K$ may be viewed as a graph \emph{embedded} in the
surface $|K|$. Accordingly, we denote its set of vertices, edges, and faces
by $V(K)$, $E(K)$, and $F(K)$, respectively. If every face of $K$ is incident to three edges,
$K$ is a \emph{triangulated surface}, or a \emph{triangulation} of $|K|$.
In this case, $K$ can be viewed as a \emph{simplicial complex}.
If $K$ is a simplicial complex and $X \subseteq V(K)$,
then $K[X]$ is the subcomplex induced by $X$, and $K\setminus X$ is the subcomplex
obtained by deleting $X$ and all incident simplices. If $L$ is a subcomplex of $K$,
then we simply write $K \setminus L$ instead of $K \setminus V(L)$.

If $K$ is a graph embedded in a surface $|K|$ without boundary, the \emph{dual graph}
$K^*$ is the graph with vertex set $F(K)$, such that $fg \in E(K^*)$ if and only if $f$ and
$g$ share an edge in $K$. The size of a face $f \in F(K)$ is defined as the number
of edges on its boundary walk, and is denoted by $d_K(f)$. Note that $d_K(f)=d_{K^*}(f^*)$.

Any polygonal surface homeomorphic to a sphere corresponds to a plane
graph via the stereographic projection. Therefore, terms such as `plane triangulation'
and `triangulation of the sphere' can be used interchangeably.
We shall make the convention to use the term `cubic plane graphs' because it is so
widespread, but refer to the dual graphs as `triangulations of the sphere' because
it reflects better our geometric viewpoint.

Given a polygonal surface $K$, the \emph{boundary} $\partial K$ is the set of all
edges in $K$ which are not incident to two triangles; the number of edges in the
boundary is denoted by $|\partial K|$. With a slight abuse of notation, $\partial K$ will also denote
the set of vertices incident to edges in $\partial K$. The set of \emph{interior} vertices
is defined as $\inter(K)=V(K)\setminus \partial K$.

Given a triangulated surface $K$, we define $\area K$ to be the number of faces in $K$,
and the \emph{combinatorial curvature} of $K$ as $\sum_{u \in \inter(K)}(6-d_K(u))$.
Recall that the \emph{Euler characteristic}
$\chi(K)$ of a polygonal surface $K$ is equal to $|V(K)|-|E(K)|+|F(K)|$. It can be shown
that $\chi$ is a topological invariant: it only depends on the surface $|K|$, not on the
polygonal decomposition of $K$. If $X$ is any contractible space, then $\chi(X)=1$, and
if $S^2$ is the standard $2$-dimensional sphere, then $\chi(S^2)=2$. The following lemma
is an easy consequence of Euler's formula and double counting, and we leave its verification
to the reader.

\begin{lemma}
\label{lem:gauss-bonnet}
  Let $K$ be a triangulated surface with (a possibly empty) boundary $\partial K$. Then
  \[
    \textstyle\sum_{v \in \inter K}(6-d(v))+\sum_{v \in \partial K}(4-d(v))=6\chi(K).
  \]
\end{lemma}

We remark that, if we multiply both sides of the equation by $\pi/3$, we obtain a discrete
version of the Gauss--Bonnet theorem (see e.g.~\cite{Lee97}), where the curvature is concentrated at the vertices. 

In order to prove Theorem~\ref{thm:OCT}, it is more convenient to work with the dual
graphs, which are characterised by the following simple lemma. The proof is an easy
exercise, which we leave to the reader.

\begin{lemma}
\label{lem:dual}
  If $G$ is a $3$-connected simple cubic plane graph with all faces of size at most $6$, then the dual
  graph $G^*$ is a simple triangulation of the sphere with all vertices of degree at
  least $3$ and at most $6$.
\end{lemma}

We will use the following
important concept from combinatorial optimisation. Given a graph $G=(V,E)$ with a distinguished
set $T$ of vertices of even cardinality, a \emph{$T$-join} of $G$ is a subset $J \subseteq E$
such that $T$ is equal to the set of odd-degree vertices in $(V,J)$. The minimum size of a
$T$-join of $G$ is denoted by $\tau(G,T)$. When $T$ is the set of odd-degree vertices of $G$,
a $T$-join is known as a \emph{postman set}. A \emph{$T$-cut} is an edge cut $\delta(X)$ such
that $|T\cap X|$ is odd. A \emph{packing} of $T$-cuts is a disjoint collection
$\delta(\mathcal F)=\{\delta(X) \mid X \in \mathcal F\}$ of $T$-cuts of $G$; the maximum
size of a packing of $T$-cuts is denoted by $\nu(G,T)$.

A family of sets $\mathcal F$ is said to be \emph{laminar} if, for every pair
$X,Y \in \mathcal F$, either $X \subseteq Y$, $Y \subseteq X$, or $X \cap Y = \emptyset$.
A $T$-cut $\delta(X)$ is \emph{inclusion-wise minimal} if no $T$-cut is properly contained
in $\delta(X)$. For more information on $T$-joins
and $T$-cuts, the reader is referred to~\cite{CCPS98, LoPl86, Sch03}.

\section{Patches and moats}
\label{sec:patches}

From now on assume that $K$ is a triangulation of the sphere with all vertices of degree at most $6$.
We define a subcomplex $L \subseteq K$ to be a \emph{patch} if in the dual complex $K^*$, the faces
corresponding to $V(L)$ form a subcomplex homeomorphic to a disc. (Equivalently, one could say
that $L \subseteq K$ is a patch if $L$ is an induced, contractible subcomplex of $K$.)
A patch $L\subseteq K$ such that $c=\sum_{u \in V(L)}(6-d_K(u))$ is called a \emph{$c$-patch}.
We remark that a $c$-patch $L$ has combinatorial curvature $c$ if and only if all vertices in the
boundary $\partial L$ have degree $6$ in $K$.
If $u \in V(K)$ has degree $6-c$, and the set $X$ of vertices at distance at most $r$ from $u$
contains only vertices of degree $6$, then the $c$-patch $K[\{u\} \cup X]$ is denoted by $D_r(c)$.
The subcomplex of the dual complex $K^*$ formed by the faces in $V(D_r(c))$ is denoted by $D^*_r(c)$;
see Figure~\ref{fig:discs}.

\begin{figure}
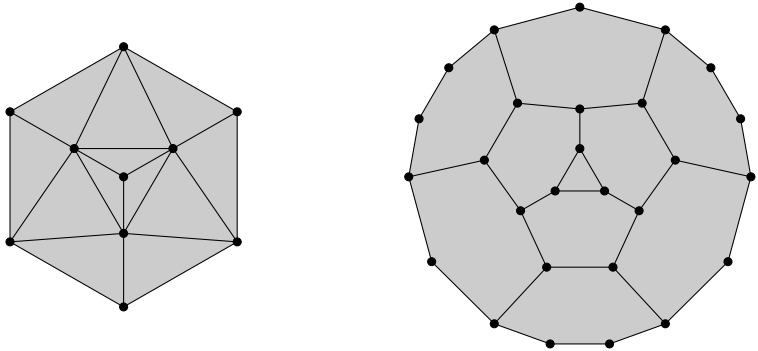

  \centering
  \begin{tikzgraph}[scale=0.75]
    \foreach\i in {1,...,3}
    {
      \path (120*\i-30:0.5) coordinate (a\i);
      \path (120*\i-30:1.2) coordinate (b\i);
    }
    \foreach\i in {1,3,5}
    {
      \path (60*\i+70:1.7) coordinate (c\i);
      \path (60*\i+60:3) coordinate (d\i);
    }
    \foreach\i in {2,4,6}
    {
      \path (60*\i+50:1.7) coordinate (c\i);
      \path (60*\i+60:3) coordinate (d\i);
    }
    \foreach\i in {1,4,7}
    {
      \path (40*\i+50:3) coordinate (e\i);
    }
    \foreach\i in {2,5,8}
    {
      \path (40*\i+60:3) coordinate (e\i);
    }
    \foreach\i in {3,6,9}
    {
      \path (40*\i+40:3) coordinate (e\i);
    }
    \draw[fill=black!20] (e1)--(d1)--(e2)--(e3)--(d2)--(e4)--(d3)--(e5)--(e6)--(d4)--(e7)--(d5)--(e8)--(e9)--(d6)--cycle;
    \draw (a1)--(a2)--(a3)--cycle
          (a1)--(b1)
          (a2)--(b2)
          (a3)--(b3)
          (b1)--(c1)--(c2)--(b2)--(c3)--(c4)--(b3)--(c5)--(c6)--cycle;
    \foreach\i in {1,...,6}
    {
      \draw (c\i)--(d\i);
    }
    \foreach\i in {1,...,3}
    {
      \draw (a\i) node[vertex] {};
      \draw (b\i) node[vertex] {};
    }
    \foreach\i in {1,...,6}
    {
      \draw (c\i) node[vertex] {};
      \draw (d\i) node[vertex] {};
    }
    \foreach\i in {1,...,9}
    {
      \draw (e\i) node[vertex] {};
    }
    \path (-8,0) coordinate (a);
    \foreach\i in {1,...,3}
    {
      \path (120*\i+30:1)+(-8,0) coordinate (b\i);
    }
    \foreach\i in {1,...,6}
    {
      \path (60*\i+30:2.3)+(-8,0) coordinate (c\i);
    }
    \draw[fill=black!20] (c1)--(c2)--(c3)--(c4)--(c5)--(c6)--cycle;
    \draw (a)--(b1)
          (a)--(b2)
          (a)--(b3)
          (b1)--(c1)
          (b1)--(c2)
          (b1)--(c3)
          (b2)--(c3)
          (b2)--(c4)
          (b2)--(c5)
          (b3)--(c5)
          (b3)--(c6)
          (b3)--(c1)
          (b1)--(b2)--(b3)--cycle;
    \draw (a) node[vertex] {};
    \foreach\i in {1,...,3}
    {
      \draw (b\i) node[vertex] {};
    }
    \foreach\i in {1,...,6}
    {
      \draw (c\i) node[vertex] {};
    }
  \end{tikzgraph}
  \caption{On the left, the $3$-patch $D_2(3)$ and on the right, the polygonal
  surface $D^*_2(3)$.}
\label{fig:discs}
\end{figure}

The following isoperimetric inequality follows from the work of
Justus~\cite[Theorem~3.2.3 and Table~3.1]{Jus07}\footnote{Gunnar Brinkmann~\cite{Bri17}
has pointed out an error in the statement and proof of \cite[Lemma~4.4]{Egg05}
on which \cite[Theorem~3.2.3]{Jus07} is based, but has sketched a different way to prove
\cite[Theorem~3.2.3]{Jus07}.}.

\begin{lemma}[Justus~\cite{Jus07}]
\label{lem:justus}
  Let $K^*$ be a polygonal surface homeomorphic to a disc, with all internal vertices of
  degree $3$ and with $n$ faces, all of size at most $6$. Let $c=\sum_{f \in F(K^*)}(6-d(f))$,
  and suppose that $c \leq 5$. Then
  \[
    |\partial K^*| \geq \sqrt{8(6-c)(n-1)+(6-c)^2}.
  \]
  Equality holds if $K^* \cong D^*_r(c)$, for some integer $r \geq 0$, and only
  if at most one face in $K^*$ has size less than $6$.
\end{lemma}

\begin{proof}
  The minimum possible values of $|\partial K^*|$ are given in~\cite[Table~3.1]{Jus07},
  for all possible numbers of hexagonal, pentagonal, square, and triangular faces.
  In each case, our bound is satisfied. Moreover, it can be checked that equality
  holds only if at most one face in $K^*$ has size less than $6$. Finally, if
  $K^* \cong D^*_r(c)$, then it can be shown that $|\partial K^*|=(6-c)(2r+1)$ and
  $f-1=(6-c)r(r+1)/2$. Hence, $|\partial K^*| = \sqrt{8(6-c)(f-1)+(6-c)^2}$.
\end{proof}

We can use Lemma~\ref{lem:justus} to deduce the following isoperimetric inequality
for triangulations. Certain special cases of the inequality were already proved by
Justus~\cite{Jus07}.

\begin{lemma}
\label{lem:isoperimetric}
  Let $K$ be a triangulation of the sphere with all vertices of degree at most $6$,
  and let $L \subseteq K$ be a patch of combinatorial curvature $c \leq 5$. Then
  \[
    |\partial L| \geq \sqrt{(6-c)\:\area L}.
  \]
  Equality holds if $L \cong D_r(c)$, for some integer $r \geq 0$, and only
  if at most one vertex in $\inter L$ has degree less than $6$.
\end{lemma}

\begin{proof}
  Put $n=|V(K)|$, and let $L^*$ be the subcomplex of $K^*$ formed by the faces
  corresponding to $V(L)$. By Lemma~\ref{lem:justus},
  \begin{equation}\label{eq:1}
    |\partial L^*| \geq \sqrt{8(6-c)(n-1)+(6-c)^2}.
  \end{equation}
  Moreover, the following two equalities were shown by Justus~\cite[equations (3.8) and (3.11)]{Jus07}
  \begin{equation}\label{eq:2}
    2(n-1)=\area L-|\partial L|,
  \end{equation}
  \begin{equation}\label{eq:3}
    |\partial L|^2 = \tfrac14|\partial(L^*)|^2+(6-c)|\partial L|-\tfrac14(6-c)^2.
  \end{equation}
  So, combining \eqref{eq:1}, \eqref{eq:2} and \eqref{eq:3} gives
  \begin{equation}\label{eq:4}
    |\partial L|^2 \geq (6-c)\:\area L.
  \end{equation}
  
  Equality holds in~\eqref{eq:4} if and only if equality holds in~\eqref{eq:1}. The
  latter is true only if at most one face in $L^*$ has size less than $6$, or equivalently,
  only if at most one vertex in $\inter L$ has degree less than $6$. For the final
  part, it is enough to note that if $L \cong D_r(c)$, then $L^* \cong D^*_r(c)$, so
  equality holds in~\eqref{eq:1} and therefore in~\eqref{eq:4}.
\end{proof}

Let $L \subseteq K$ be a patch. A \emph{moat} of width $1$ in $K$ surrounding $L$
is the set $\Mt^1(L)$ of all the faces in $F(K)\setminus F(L)$ with at
least one vertex in $V(K)$. More generally, we can define a moat of width $w$ in $K$ surrounding
$L$ recursively as $\Mt^w(L)=\Mt^1(\Mt^{w-1}(L)\cup L)$. With a slight abuse of notation,
$\Mt^w(L)$ will also denote the subcomplex of $K$ formed by the faces in $\Mt^w(L)$.
If $L$ is a $c$-patch, then $\Mt^w(L)$ is a \emph{$c$-moat} of width $w$ surrounding $L$.
See Figure~\ref{fig:patch-moat} for an example of a moat.

\begin{figure}
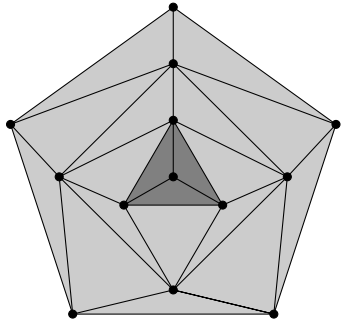

  \centering
  \begin{tikzgraph}[scale=0.75]
    \path (90:1) coordinate (a1);
    \path (210:1) coordinate (a2);
    \path (330:1) coordinate (a3);
    \path (90:2) coordinate (b1);
    \path (180:2) coordinate (b2);
    \path (270:2) coordinate (b3);
    \path (0:2) coordinate (b4);
    \foreach\i in {1,...,5}
    {
      \path (72*\i+18:3) coordinate (c\i);
    }
    \draw[fill=black!20] (c1)--(c2)--(c3)--(c4)--(c5)--cycle;
    \draw[fill=black!50] (a1)--(a2)--(a3)--cycle;
    \draw (0,0)--(a1)
          (0,0)--(a2)
          (0,0)--(a3)
          (a1)--(b1)
          (a1)--(b2)
          (a2)--(b2)
          (a2)--(b3)
          (a3)--(b3)
          (a3)--(b4)
          (a1)--(b4)
          (c1)--(b1)--(c2)--(b2)--(c3)--(b3)--(c4)
          (b3)--(c4)--(b4)--(c5)--(b1);
    \draw (b1)--(b2)--(b3)--(b4)--cycle;
    \draw (0,0) node[vertex] {};
    \foreach\i in {1,...,3}
    {
      \draw (a\i) node[vertex] {};
    }
    \foreach\i in {1,...,4}
    {
      \draw (b\i) node[vertex] {};
    }
    \foreach\i in {1,...,5}
    {
      \draw (c\i) node[vertex] {};
    }
  \end{tikzgraph}
  \caption{A $5$-patch $L$ of area $3$ (shaded
  in dark grey) surrounded by a $5$-moat $\Mt^2(L)$ of width $2$ and area $16$
  (shaded in light grey).}
\label{fig:patch-moat}
\end{figure}

Under certain conditions, the area of a $c$-moat $\Mt^w(L)$ can be bounded
in terms of $c$, $w$, and $\area L$.

\begin{lemma}
\label{lem:perimeter}
  Let $K$ be a triangulation of the sphere with maximum degree at most $6$,
  and suppose $L \subseteq K$ is a $c$-patch, for some $0<c<6$.
  If $L\cup \Mt^i(L)$ is a $c$-patch, for every $0\leq i \leq w-1$, then
  \[
    \area{\Mt^w(L)} \geq (6-c)w^2+2w\sqrt{(6-c)\:\area L}.
  \]
  Equality holds if $L \cong D_r(c)$, for some integer $r \geq 0$, and only
  if at most one vertex in $\inter L$ has degree less than $6$.
\end{lemma}

\begin{proof}
  As $L$ is contractible, its Euler characteristic is $\chi(L)=1$.
  We have
  \begin{align*}\label{eq:boundary}
    c &= \sum_{u \in V(L)}(6-d_{K}(u))\\  
      &= \sum_{u \in \inter L}(6-d_L(u))+\sum_{u \in \partial L}(6-d_L(u))-\area{\Mt^1(L)}\\
      &= \sum_{u \in \inter L}(6-d_L(u))+\sum_{u \in \partial L}(4-d_L(u))+2|\partial L|-\area{\Mt^1(L)}.
  \end{align*}
  Hence, by Lemma~\ref{lem:gauss-bonnet},
  \begin{equation}\label{eq:boundary}
    2|\partial L|+6-c = \area{\Mt^1(L)}.
  \end{equation}
  
  The dual complex $K^*$ is homeomorphic to the sphere and the subcomplex $L^*$
  formed by the faces corresponding to $V(L)$ is homeomorphic to a disc, so by the
  Jordan--Schoenflies theorem, the subcomplex formed by the faces corresponding
  to $V(K)\setminus V(L)$ is also homeomorphic to a disc. Hence, $K \setminus L$ is
  also a patch. Moreover, $K$ has Euler characteristic $\chi(K)=2$, so by
  Lemma~\ref{lem:gauss-bonnet}, $\sum_{u \in V(K)}(6-d(u))=12$. Therefore,
  $\sum_{u \in V(K\setminus L)}(6-d_K(u))=12-c$, i.e., $K \setminus L$ is a
  $(12-c)$-patch. Applying~\eqref{eq:boundary} to $L$ and to
  $K\setminus L$,
  \begin{align*}
    2|\partial L|+6-c &= \area{\Mt^1(L)} \\
                      &= \area{\Mt^1(K\setminus L)} \\
                      &= 2|\partial(K\setminus L)|+6-(12-c).
  \end{align*}
  Hence, $|\partial(L \cup \Mt^1(L))|=|\partial(K\setminus L)|=|\partial L|+6-c$, so by
  induction, and the fact that $L \cup \Mt^i(L)$ is a patch for all $0 \leq i \leq w-1$,
  \begin{equation}\label{eq:Mt^i(L)}
    |\partial(L \cup \Mt^i(L))|=|\partial L|+(6-c)i.
  \end{equation}
  By~\eqref{eq:boundary} and~\eqref{eq:Mt^i(L)},
  \begin{align*}
    \area{\Mt^1(L \cup \Mt^i(L))} &= 2|\partial(L \cup \Mt^i(L))|+6-c \\
                                  &= 2(|\partial L|+(6-c)i)+6-c \\
                                  &= 2|\partial L|+(6-c)(2i+1),
  \end{align*}
  so the area of $\Mt^w(L)$ is
  \begin{align*}
    \area{\Mt^w(L)} &= \sum_{i=0}^{w-1}\area{\Mt^1\left(L \cup \Mt^i(L)\right)}\\
                    &= \sum_{i=0}^{w-1}\left(2|\partial L|+(6-c)(2i+1)\right)\\
                    &= 2w|\partial L|+(6-c)w^2.
  \end{align*}
  The combinatorial curvature of $L$ is at most $c$, so by Lemma~\ref{lem:isoperimetric},
  \[
    |\partial L| \geq \sqrt{(6-c)\:\area L},
  \]
  with equality if $L \cong D_r(c)$, for some integer $r \geq 0$, and only
  if at most one vertex in $\inter L$ has degree less than $6$.
\end{proof}

\section{Packing odd cuts in triangulations of the sphere with maximum degree at most $6$}
\label{sec:T-cut}

We now relate certain special types of packings of $T$-cuts to packings of
$1$-, $3$- and $5$-moats.

\begin{lemma}
\label{lem:moat-packing}
  Let $K$ be a triangulation of the sphere with all vertices of degree at most $6$,
  and let $T$ be the set of odd-degree vertices in $K$. There exists a family
  $\mathcal F$ on $V(K)$ and a vector $w \in \NN^{|\mathcal F|}$ with the following
  properties.
  \begin{enumerate}[label=\textnormal{(M\arabic*)}]
    \item \label{M:packing} $\mathcal M=\bigcup_{X \in \mathcal F}\Mt_K^{w_X}(X)$ is a packing of
      moats in $K$;
    \item \label{M:width} The total width of $\mathcal M$ is $\sum_{X \in \mathcal F}w_X=\nu(K,T)$;
    \item \label{M:patch} For every $X \in \mathcal F$, the subcomplex $K[X]$ is a patch;
    \item \label{M:1-3-5} Every $\Mt^{w_X}(X) \in \mathcal M$ is a $1$-, $3$-, or $5$-moat in $K$;
    \item \label{M:minimal} If $X$ is an inclusion-wise minimal element in $\mathcal F$, then $|X|=1$;
    \item \label{M:laminar} $\mathcal F$ is laminar.
  \end{enumerate}
\end{lemma}

\begin{proof}
  Consider a packing $\delta(\mathcal F')$ of inclusion-wise minimal $T$-cuts in $K$ of size
  of $\nu(K,T)$. Note that $\sum_{u \in X}(6-d_K(u))$ is odd, for every $X \in \mathcal F$.
  Since $\sum_{u \in X}(6-d_K(u))=12$ and $\delta(X)=\delta(V(K)\setminus X)$, we can assume that
  $\sum_{u \in X}(6-d_K(u))\leq 5$; otherwise we could replace $X$ by $V(K)\setminus X$ in
  $\delta(\mathcal F')$. Finally, we can also assume that,
  subject to the above conditions, $\mathcal F'$ minimises $\sum_{X \in \mathcal F'} |X|$.
  
  We remark that $\mathcal F'$ is a laminar family. Indeed, suppose that $X,Y \in \mathcal F'$,
  $X \cap Y \neq \varnothing$, $X \not\subseteq Y$ and $Y \not\subseteq X$. Then
  $\Mt^1(X)\cap \Mt^1(Y)\neq \varnothing$, so there is a face $\{u,v,w\}$ of $K$ in
  $\Mt^1(X)\cap \Mt^1(Y)$. Since
  \[
    |\delta(X)\cap \{uv,uw,vw\}|=|\delta(Y)\cap \{uv,uw,vw\}| = 2,
  \]
  it follows that $\delta(X)\cap \delta(Y)\neq \varnothing$, contradicting the fact that
  $\mathcal F'$ is a packing of $T$-cuts. Hence, $\mathcal F'$ is laminar.

  We summarise the properties of the family $\mathcal F'$ below.
  \begin{enumerate}[label=\textnormal{(P\arabic*)}]
    \item \label{P:packing} $\delta(\mathcal F')$ is a packing of $T$-cuts;
    \item \label{P:size} $|\mathcal F'|=\nu(K,T)$;
    \item \label{P:inclusion-wise-minimal} $\delta(X)$ is an inclusion-wise minimal cut,
      for every $X \in \mathcal F'$;
    \item \label{P:1-3-5} $\sum_{u \in X}(6-d_K(u)) \in \{1,3,5\}$ for all $X \in \mathcal F'$;
    \item \label{P:minimal} Subject to \ref{P:packing}--\ref{P:1-3-5}, $\mathcal F'$
      minimises $\sum_{X \in \mathcal F'} |X|$;
    \item \label{P:laminar} $\mathcal F'$ is laminar.
  \end{enumerate}
  
  We let $\mathcal F$ be the subfamily of $\mathcal F'$ consisting of the
  elements $X \in \mathcal F'$ such that
  \[
    \sum_{u \in Y}(6-d_K(u))<\sum_{u \in X}(6-d_K(u)),
  \]
  for every $Y \in \mathcal F$ such that $Y \subseteq X$.
  For each $X \in \mathcal F$, let
  \[
    \mathcal F'_X=\{Y \in \mathcal F' : X \subseteq Y, \sum_{u \in Y}(6-d_K(u))=\sum_{u \in X}(6-d_K(u))\},
  \]
  and let $w_X=|F'_X|$.
  
  To prove~\ref{M:packing}, we use an argument very similar to the one we used to
  prove~\ref{P:laminar}.
  Clearly, for every $X \in \mathcal F$, $\Mt^{w_X}(X)=\bigcup_{Y \in \mathcal F'_X} \Mt^1(Y)$
  is a moat around $X$ of width $w_X$. Let $X,Y \in \mathcal F$, and suppose that
  $\Mt^{w_X}(X) \cap \Mt^{w_Y}(Y) \neq \varnothing$. Then there exists a face
  $\{u,v,w\} \in F(K)$ and sets $X',Y' \in \mathcal F'$ such that $X \subseteq X'$,
  $Y \subseteq Y'$, and
  \[
    |\delta(X')\cap \{uv,uw,vw\}|=|\delta(Y')\cap \{uv,uw,vw\}| = 2.
  \]
  But then $\delta(X') \cap \delta(Y') \neq \varnothing$, so by~\ref{P:packing},
  $X'=Y'$. Hence, by the construction of $\mathcal F$, $X=Y$. This proves~\ref{M:packing}.
  
  To prove~\ref{M:width}, it suffices to note that $\sum_{X \in \mathcal F}w_X=|\mathcal F'|=\nu(K,T)$
  by~\ref{P:size}.
  
  The property~\ref{M:patch} follows immediately from~\ref{P:inclusion-wise-minimal};
  indeed, since $\delta(X)$ is
  an inclusion-wise minimal cut, the dual edges form a cycle, so by the
  Jordan--Schoenflies theorem, the subcomplex of $K^*$ formed by the faces in
  $X$ is homeomorphic to a disc, so $K[X]$ is a patch.

  Since $\mathcal F \subseteq \mathcal F'$, \ref{M:1-3-5} follows immediately
  from~\ref{P:1-3-5} and ~\ref{M:laminar} follows immediately from~\ref{P:laminar}.
  
  To prove~\ref{M:minimal}, let $X$ be an inclusion-wise minimal element
  of $\mathcal F$. By the definition of $\mathcal F$, $X$ is also an inclusion-wise
  minimal element of $\mathcal F'$. Since $\sum_{u \in X}d(u)$ is odd, at least one
  vertex in $X$ has odd degree. If $|X|>1$, let $u$ be a vertex of odd degree in $X$,
  and let $\mathcal F''=(\mathcal F'\setminus X) \cup \{u\}$. Then $\mathcal F''$ satisfies
  \ref{P:packing}--\ref{P:1-3-5}, but $\sum_{X \in \mathcal F''} |X|<\sum_{X \in \mathcal F'} |X|$,
  contradicting~\ref{P:minimal}.
\end{proof}

Lemmas~\ref{lem:perimeter} and~\ref{lem:moat-packing} can be used to prove
the following upper bound on the maximum size of a packing of odd cuts in
spherical triangulations with all vertices of degree at most $6$,
which may be of independent iterest. By taking the planar dual, we also get
an upper bound on $\nu_{\odd}$ for the class of $3$-connected cubic plane
graphs with all faces of size at most $6$.

\begin{theorem}\label{thm:packing}
  Let $K$ be a triangulation of the sphere with maximum degree
  at most $6$. If $T$ is the set of odd-degree vertices of $K$, then
  \[
    \nu(K,T) \leq \sqrt{\tfrac 15\textstyle\sum_{u \in T}(6-d(u))\;\area K}.
  \]
  In particular, $\nu(K,T) \leq \sqrt{12\;\area K/5}$ always holds,
  with equality if and only if all vertices have degree $5$ and $6$, $\area K=60k^2$
  for some $k \in \NN$, and $\Aut(K) \cong I_h$.
\end{theorem}

\begin{proof}
  Let $\mathcal M=\bigcup_{X \in \mathcal F}\Mt_K^{w_X}(X)$ be a packing of $1$-, $3$- and
  $5$-moats in $K$ of total width $\sum_{X \in \mathcal F}w_X=\nu(K,T)$, as
  guaranteed by Lemma~\ref{lem:moat-packing}. Let $m_c$ be the total area of
  $c$-moats of $\bigcup_{X \in \mathcal F}\Mt_K^{w_X}(X)$, where
  $c\in\{1,3,5\}$. Define the incidence vectors $r, s, t \in \mathbb R^{|T|}$  as
  follows: for every $u \in T$, let $r_u$, $s_u$, $t_u$ be the width of the $1$-moat,
  $3$-moat and $5$-moat surrounding $u$, respectively.

  Define the inner product $\langle \cdot,\cdot \rangle$ on $\RR^{|T|}$ by
  $\left\langle x,y \right\rangle = \sum_{u\in T}(6-d(u)) x_uy_u$ and the norm
  $\|\cdot\|$ by $\|x\|=\left\langle x,x \right\rangle$. With this inner product,
  the total width of $1$-, $3$- and $5$-moats in
  $\bigcup_{X \in \mathcal F}\Mt_K^{w_X}(X)$ can be expressed as
  $\left\langle r,1\right\rangle$, $\tfrac13 \left\langle s,1\right\rangle$, and
  $\tfrac15 \left\langle t,1\right\rangle$, respectively. Therefore,
  \begin{equation}\label{eq:moat_packing}
    \nu(K,T)=\sum_{X \in \mathcal F}w_X=\left\langle r+\tfrac13s+\tfrac15t,1\right\rangle.
  \end{equation}

  To prove the inequality in Theorem~\ref{thm:packing}, we compute lower bounds on
  $m_1$, $m_3$ and $m_5$ in terms of the vectors $r$, $s$ and $t$, and then use
  the fact that the moats are disjoint, so the sum $m_1+m_3+m_5$ cannot exceed $f$,
  the number of faces of $K$. Simplifying the inequality gives the desired bound.
  
  To bound $m_1$, recall that by property~\ref{M:minimal} of Lemma~\ref{lem:moat-packing},
  every $1$-moat in $\mathcal M$ is of the form $\Mt_{K}^{r_u}(u)$, where $u$ is a
  $5$-vertex in $K$. By Lemma~\ref{lem:perimeter},
  \begin{equation}
  \label{eq:one1-moat}
    \area{\Mt_{K}^{r_u}(u)} = (6-(6-d(u)))r_u^2= 5r_u^2,
  \end{equation}
  and summing over all $1$-moats gives the equality
  \begin{equation}
  \label{eq:1-moat}
    m_1 = 5 \sum_{u\in T}\left(6-d(u)\right)r_u^2 = 5\|r\|^2.
  \end{equation}

  To bound $m_3$, let $\Mt_{K}^{s_u}(X)$ be a non-empty $3$-moat in $\mathcal M$, for
  some $u\in T \cap X$. By the laminarity of $\mathcal M$, the $3$-patch $K[X]$ contains
  the (possibly empty) $1$-moats $\Mt_{K}^{r_u}(u)$, for all $5$-vertices
  $u \in T \cap X$. All the moats are pairwise disjoint, so by~\eqref{eq:one1-moat}
  and the Cauchy--Schwarz inequality,
  \begin{align*}
    \area{K[X]}
      &\geq \sum_{u \in T \cap X}\area{\Mt_{K}^{r_u}(u)}\\
      &\geq 5 \sum_{u\in T \cap X}\left(6-d(u)\right)r_u^2\\
      &\geq \frac{5\left(\sum_{u\in T \cap X}\left(6-d(u)\right)r_u\right)^2}%
            {\sum_{u \in T \cap X} (6-d(u))}\\
      &\geq \frac{5}{3} \left(\sum_{u\in T \cap X}\left(6-d(u)\right)r_u\right)^2.
  \end{align*}
  Hence, by Lemma~\ref{lem:perimeter},
  \begin{align}
    \area{\Mt_K^{s_u}(X)}
      &\geq 3s_u^2+2s_u\sqrt{3\:\area{K[X]}} \notag\\
      &\geq \sum_{u\in T  \cap X}\left(6-d(u)\right)s_u^2+2\sqrt{5} \sum_{u \in T%
        \cap X}\left(6-d(u)\right)r_us_u \label{eq:one3-moat}.
  \end{align}
  Summing over all $3$-moats gives the inequality
  \begin{equation}
  \label{eq:3-moat}
    m_3 \geq \|s\|^2+2\sqrt 5\langle r,s \rangle.
  \end{equation}

  To bound $m_5$, let $\Mt_{K}^{t_u}(Y)$ be a non-empty $5$-moat in $\mathcal M$, for
  some $u\in T\cap Y$. By the laminarity of $\mathcal M$, the $5$-patch $K[Y]$ contains at most one
  non-empty $3$-moat $\Mt_{K}^{s_u}(X)$ of $\mathcal M$. All the moats are pairwise disjoint,
  so by~\eqref{eq:one1-moat}, \eqref{eq:one3-moat} and the Cauchy--Schwarz
  inequality,
  \begin{align*}
    \area{K[Y]}
      &\geq \sum_{u\in T \cap Y}\area{\Mt_K^{r_u}(u)}+%
            \sum_{u\in T \cap Y}\area{\Mt_K^{s_u}(X)}\\
      &\geq 5 \sum_{u\in T \cap Y}\left(6-d(u)\right)r_u^2+\sum_{u \in T \cap Y}%
            \left(6-d(u)\right)\left(2\sqrt 5r_us_u+s_u^2\right)\\
      &=    5 \sum_{u\in T \cap Y}\left(6-d(u)\right)\left(r_u+\tfrac1{\sqrt{5}}%
            s_u\right)^2\\
      &\geq \frac{5\left(\sum_{u \in T \cap Y}\left(6-d(u)\right)\left(r_u+%
            \tfrac 1{\sqrt 5}s_u\right) \right)^2}{\sum_{u \in T \cap Y} (6-d(u))}\\
      &=    \left(\sum_{u \in T \cap Y}\left(6-d(u)\right)\left(r_u+\tfrac1{%
            \sqrt 5}s_u\right) \right)^2.
  \end{align*} 
  Using Lemma~\ref{lem:perimeter},
  \begin{align*}
    \area{\Mt_K^{t_u}(Y)}
      & \geq t_u^2+2t_u\sqrt{\area{K[Y]}} \\
      & \geq \tfrac15\sum_{u \in T \cap Y}(6-d(u))t_u^2+2t_u\sum_{u \in T \cap Y}
        \left(6-d(u)\right)\left(r_u+\tfrac{1}{\sqrt 5}s_u\right) \\
      &=\sum_{u \in T \cap Y}\left(6-d(u)\right)\left(\tfrac15t_u^2+2r_ut_u+
        \tfrac{2}{\sqrt 5}s_ut_u\right),
  \end{align*}
  with equality only if $t_u=0$, because $K$ is a simple triangulation of the sphere,
  and as such has no vertex of degree $1$. Summing over all $5$-moats gives the inequality
  \begin{equation}
  \label{eq:5-moat}
    m_5 \geq \tfrac15\|t\|^2+2\langle r,t \rangle+\tfrac{2}{\sqrt 5}\langle s,t \rangle,
  \end{equation}
  with equality only if $t=0$.
  
  The moats are disjoint, so by
  inequalitites~\eqref{eq:1-moat}, \eqref{eq:3-moat} and~\eqref{eq:5-moat},
  \begin{align*}
   \area K &\geq m_1+m_3+m_5\\
           &\geq 5\|r\|^2 + \|s\|^2 + \tfrac15\|t\|^2 + 2\sqrt 5\langle r,s \rangle +
                 2\langle r,t \rangle + \tfrac{2}{\sqrt 5}\langle s,t \rangle\\
           &=    \left\|\sqrt 5r + s + \tfrac 1{\sqrt 5}t\right\|^2.
  \end{align*}
  Hence, by the Cauchy--Schwarz inequality and~\eqref{eq:moat_packing},
  \begin{align}
    \sqrt{\tfrac15\textstyle\sum_{u \in T}(6-d(u))\;\area K}
      & \geq \sqrt{\textstyle\sum_{u \in T}(6-d(u))} \left\|r + \tfrac 1{\sqrt 5}s + \tfrac 15t\right\| \notag\\
      & \geq \left\langle r + \tfrac 1{\sqrt 5} s + \tfrac 15 t,1\right
             \rangle \label{eq:cauchy-schwarz}\\
      & \geq \left\langle r + \tfrac 13 s + \tfrac 15 t,1\right\rangle
             \label{eq:3-vertices}\\
      & =    \nu(K,T). \notag
  \end{align}
  This completes the proof of the first part of Theorem~\ref{thm:packing}.

  To prove the inequality $\nu(K,T)\leq\sqrt{12\;\area K/5}$, it suffices
  to observe that $\textstyle{\sum}_{u\in T}(6-d(u))\leq 12$
  by Lemma~\ref{lem:gauss-bonnet}. Now suppose that
  $\nu(K,T)=\sqrt{12\;\area K/5}$. By Lemma~\ref{lem:moat-packing}, there
  exists a packing $\mathcal M=\bigcup_{X \in \mathcal F}\Mt_K^{w_X}(X)$
  of $1$-, $3$- and $5$-moats in $K$ of total width $\sqrt{12\;\area K/5}$.
  Then $\sum_{u \in T}(6-d(u))=12$, i.e., all
  vertices of degree less than $6$ have odd degree, namely, $3$ or $5$.
  Equality holds in~\eqref{eq:5-moat} and in~\eqref{eq:3-vertices}, so $t=s=0$.
  Furthermore, equality holds in~\eqref{eq:cauchy-schwarz}, so there is a natural
  number $k \geq 1$ such that $r_u=k$ for every $u \in T$. Therefore, every $u \in T$ has
  degree $5$, so $|T|=12$. By Lemma~\ref{lem:perimeter} each moat
  $\Mt_K^{k}(u) \in \mathcal M$ has area $5k^2$, so $\area K=12\cdot5 k^2=60k^2$.   
  Hence, $K$ is the union of twelve face-disjoint $1$-moats $\Mt^k(u)$,
  for $u \in T$ (see Figure~\ref{fig:disc}). Each $\Mt^k(u)$ can be identified
  with a face of a regular dodecahedron, which shows that $\Aut(K)$ contains
  a subgroup isomorphic to $I_h$. On the other hand, the dual graph of $K$ is
  a fullerene graph, and it can be shown (see e.g.~\cite{DDF09}) that the largest
  possible automorphism group of a fullerene graph is isomorphic to $I_h$. Hence,
  $\Aut(K) \cong I_h$.
    
  Conversely, suppose $K$ is a triangulation of the sphere with $\area K=60k^2$,
  all vertices of degree $5$ and $6$, and $\Aut(K) \cong I_h$. Then it can be
  shown (see~\cite{Cox71,Gol37}) that $K$ can be constructed by pasting
  triangular regions of the (infinite) $6$-regular triangulation of the plane
  into the faces of a regular icosahedron (this is sometimes known in the literature
  as the \emph{Goldberg--Coxeter construction}). The construction is uniquely
  determined by a $2$-dimensional vector $(i,j) \in \ZZ^2$, known as the
  \emph{Goldberg--Coxeter vector} (see Figure~\ref{fig:goldberg-coxeter}).
  Since $\Aut(K)\cong I_h$, we must have $j=0$ or $j=i$.
  The area of $K$ is given by the formula $\area K = 20(i^2 + ij + j^2)$.
  The condition $\area K=60k^2$ implies that the Goldberg--Coxeter vector of $K$ is
  $(k,k)$, which means that the distance between
  any pair of $5$-vertices in $K$ is at least $2k$. Therefore,
  $\bigcup_{u \in T}\Mt^k(u)$ is a packing of $1$-moats of total width
  $12k=12\sqrt{\area K/60}=\sqrt{12\;\area K/5}$, so $\nu(K,T)\geq\sqrt{12\;\area K/5}$.
\end{proof}

\begin{figure}
  \centering
  \begin{tikzpicture}[scale=0.75]
    \def\k{3}
    \foreach \i in {1,...,5}
    {
        \foreach \j in {1,...,\k}
        {
          \draw[rotate=72*\i] (72*3+18:\j)--(72*4+18:\j)
                (72*2+18:\k-\j)--++(0,\j)
                (72*5+18:\k-\j)--++(0,\j);
        }
    }
    \draw (0.2,0.3) node {$u$};
  \end{tikzpicture}
  \caption{The $1$-moat $\Mt^3(u)$ in the proof of Theorem~\ref{thm:packing}.}
\label{fig:disc}
\end{figure}
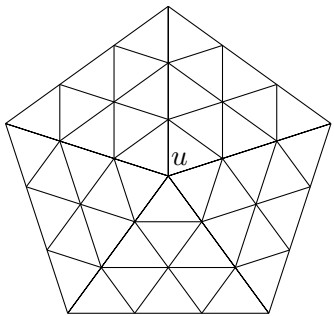

\begin{figure}
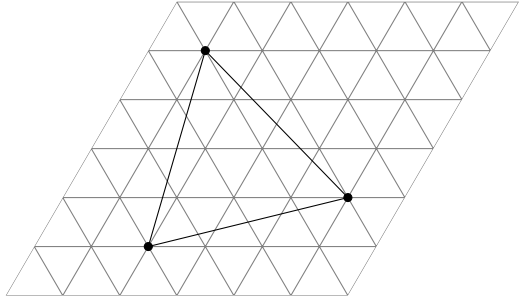

  \centering
  \begin{tikzgraph}[scale=0.75,yscale=sqrt(.75),xslant=.5]
    \def\n{6}
    \clip(0,0)--(\n,0)--(\n,\n)--(0,\n)--cycle;
    \draw[gray] (0,0) grid(\n,\n)[xslant=-1]
                (0,0) grid[ystep=\n](2*\n,2*\n);    
    \draw (2,1)--(5,2)--(1,5)--cycle;
    \draw (2,1) node[vertex] {}
          (5,2) node[vertex] {}
          (1,5) node[vertex] {};
  \end{tikzgraph}
  \caption{The Goldberg--Coxeter construction with Goldberg--Coxeter vector $(3,1)$
           (to go from the bottom left vertex to the vertex on the right, take three
           steps to the right, then make a $60$ degree left turn and take one more
           step).}
\label{fig:goldberg-coxeter}
\end{figure}

\section{Proof of Theorem~\ref{thm:OCT}}\label{sec:T-join}

Given a triangulation $K$ of the sphere, we construct the \emph{refinement} $\hat K$
as follows. First, we subdivide each edge of $K$, that is, we replace it by an
internally disjoint path of length $2$, and then we add three new edges inside every face,
incident to the three vertices of degree $2$. (For an illustration, see Figure~\ref{fig:refinement}.)
Therefore, every face of $K$ is divided into four faces of $\hat K$.
Observe that all the vertices in $V(\hat K)\setminus V(K)$ have degree $6$ in
$\hat K$, so if $T$ is the set of odd-degree vertices of $K$, then $T$
is also the set of odd-degree vertices of $\hat K$.

\begin{figure}
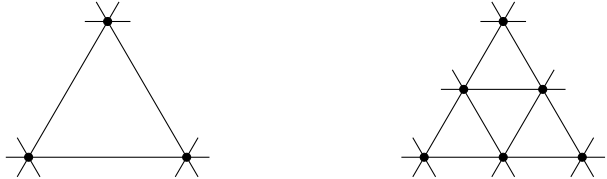

  \centering
  \begin{tikzgraph}[scale=0.6]
    \path (90:2) coordinate (a1);
    \path (210:2) coordinate (a2);
    \path (330:2) coordinate (a3);
    \draw (a1)--(a2)--(a3)--cycle;
    \foreach\i in {1,...,3}
    {
      \draw (a\i) node[vertex] {};
    }
    \foreach\i in {0,...,3}
    {
      \draw (a1)--+(60*\i:0.5);
    }
    \foreach\i in {2,...,5}
    {
      \draw (a2)--+(60*\i:0.5);
    }
    \foreach\i in {4,...,7}
    {
      \draw (a3)--+(60*\i:0.5);
    }
  \end{tikzgraph}
  \hfil
  \begin{tikzgraph}[scale=0.6]
    \path (90:2) coordinate (a1);
    \path (210:2) coordinate (a2);
    \path (330:2) coordinate (a3);
    \path (150:1) coordinate (s1);
    \path (270:1) coordinate (s2);
    \path (30:1) coordinate (s3);
    \draw (a1)--(a2)--(a3)--cycle;
    \draw (s1)--(s2)--(s3)--cycle;
    \foreach\i in {1,...,3}
    {
      \draw (a\i) node[vertex] {};
    }
    \foreach\i in {0,...,3}
    {
      \draw (a1)--+(60*\i:0.5);
    }
    \foreach\i in {2,...,5}
    {
      \draw (a2)--+(60*\i:0.5);
    }
    \foreach\i in {4,...,7}
    {
      \draw (a3)--+(60*\i:0.5);
    }
    \foreach\i in {0,...,1}
    {
      \draw (s1)--+(120+60*\i:0.5);
    }
    \foreach\i in {2,...,3}
    {
      \draw (s2)--+(120+60*\i:0.5);
    }
    \foreach\i in {0,...,1}
    {
      \draw (s3)--+(60*\i:0.5);
    }
    \foreach\i in {1,...,3}
    {
      \draw (s\i) node[vertex] {};
    }
  \end{tikzgraph}
  \caption{A face of a triangulation and its refinement.}
\label{fig:refinement}
\end{figure}

The following lemma was proved in~\cite{FKS12} using a theorem of
Seymour~\cite{Sey81}.

\begin{lemma}\label{lem:refinement}
  If $K$ is a triangulation of the sphere and $T \subseteq V(K)$ is a subset
  of even cardinality, then $\tau(K,T)=\tfrac12 \nu(\hat K,T)$.
\end{lemma}

Theorem~\ref{thm:packing} and Lemma~\ref{lem:refinement} immediately
give the following tight upper bound on the minimum size of a postman set in a plane
triangulation with maximum degree $6$.

\begin{theorem}
\label{thm:T-join}
  Let $K$ be a triangulation of the sphere with $f$ faces and maximum degree at most $6$.
  If $T$ is the set of odd-degree vertices of $G$, then
  \[
    \tau(K,T) \leq \sqrt{\textstyle \tfrac15 \sum_{u \in T}(6-d(u))\;\area K}.
  \]
  In particular,
  $\tau(K,T) \leq \sqrt{12\;\area K/5}$ always holds, with equality if and only if
  all vertices have degree $5$ and $6$, $\area K=60k^2$ for some $k \in \NN$, and
  $\Aut(G) \cong I_h$.
\end{theorem}

\begin{proof}
  Let $K$ be a triangulation of the sphere with maximum degree at most $6$, and
  let $\hat K$ be its refinement; observe that $\area{\hat K}=4\;\area K$. By
  Lemma~\ref{lem:refinement} and Theorem~\ref{thm:packing},
  \[
    \tau(K,T)=\tfrac12 \nu(\hat K,T)\leq \sqrt{\tfrac15\textstyle{\sum}_{u\in T}(6-d(u))\;\area K}\leq \sqrt{12\;\area K/5},
  \]
  as required.
  
  If $\tau(K,T)=\sqrt{12\;\area K/5}$, then $\nu(\hat K,T) = \sqrt{12\cdot 4\;\area K/5}$,
  so by the second part of Theorem~\ref{thm:packing}, all vertices in $\hat K$
  have degree $5$ and $6$, and this must clearly hold in $K$. Furthermore,
  $4\;\area K=60\hat k^2$, for some $\hat k \in \NN$, so $\area K=15\hat k^2$.
  Since $\area K$ is even, $\hat k=2k$, for some $k \in \NN$, so $\area K=60k^2$.  
  We also have $\Aut(K)\cong \Aut(\hat K) \cong I_h$.
  
  Conversely, suppose $K$ is a triangulation of the sphere with $\area K=60k^2$,
  all vertices of degree $5$ and $6$, and $\Aut(K)\cong I_h$. By
  Theorem~\ref{thm:packing} $\nu(K,T)=\sqrt{12\;\area K/5}$, so
  $\tau(K,T)=\sqrt{12\;\area K/5}$.
\end{proof}

Theorem~\ref{thm:OCT} follows from Theorem~\ref{thm:T-join} by taking the planar
dual.

\begin{proof}[Proof of Theorem~\ref{thm:OCT}]
  Let $G$ be a $3$-connected cubic plane graph on $n$ vertices with all faces of
  size at most $6$, with $p$ pentagonal and $t$ triangular faces.
  By Lemma~\ref{lem:dual}, the dual graph $G^*$ is a plane triangulation
  with $\area{G^*}=n$ and all vertices of degree at most $6$, having exactly $p$ vertices
  of degree $5$ and $t$ vertices of degree $3$. Let $T$ be the set of vertices
  of odd degree, $J^*$ a minimum $T$-join of $G^*$, and $J$ the set of edges of $G$
  which correspond to $J^*$. Since $G^*\setminus J^*$ has no odd-degree vertices,
  $G\setminus J=(G^*\setminus J^*)^*$ has no odd faces, so is bipartite. By Theorem~\ref{thm:T-join},
  \[
    |J|=|J^*| \leq \sqrt{\tfrac15\textstyle{\sum}_{u\in T}(6-d(u))n} = \sqrt{\tfrac15(p+3t)n}.
  \]
  In particular, $|J|\leq \sqrt{12n/5}$, with equality if and only if all faces
  have size $5$ and $6$, $n=60k^2$ for some $k \in \NN$, and $\Aut(G) \cong I_h$.
\end{proof}

\section{Consequences for max-cut and independence number}\label{sec:max-cut}

A classic problem in combinatorial optimisation, known as \emph{max-cut},
asks for the maximum size of an edge-cut in a given graph. This problem is
known to be NP-hard, even when restricted to triangle-free cubic graphs~\cite{Yan78}.
However, for the class of planar graphs, the problem
can be solved in polynomial time using standard tools from combinatorial
optimisation (namely $T$-joins), as observed by Hadlock~\cite{Had75}.
Cui and Wang~\cite{CuiWan09} proved that every planar, cubic graph on
$n$ vertices has a cut of size at least $39n/32-9/16$, improving an earlier
bound of Thomassen~\cite{Tho06}. However, when the face size is bounded by
$6$, we get the following improved bound.

\begin{corollary}
\label{cor:max-cut}
  If $G$ is a $3$-connected cubic plane graph on $n$ vertices with all faces of size at
  most $6$, with $p$ pentagonal and $t$ triangular faces, then $G$ has a
  cut of size at least
  \[
    3n/2-\sqrt{(p+3t)n/5}.
  \]
  In particular, $G$ has a cut of size at least $3n/2-\sqrt{12n/5}$, with
  equality if and only if all faces have size $5$ and $6$, $n=60k^2$ for some
  $k \in \NN$, and $\Aut(G) \cong I_h$.
\end{corollary}

\begin{proof}[Proof of Corollary~\ref{cor:max-cut}]
  Let $G$ be a $3$-connected cubic plane graph on $n$ vertices with all faces
  of size at most $6$. Let $J \subseteq E(G)$ be an odd cycle transversal, and
  let $X$ be a colour class of $G\setminus J$. Then
  $|\delta_G(X)| = 3n/2-|J|$. By Theorem~\ref{thm:OCT}, we can always find $J$
  such that $|J|\leq \sqrt{12n/5}$, with equality if and only if all faces have
  size $5$ and $6$, $n=60k^2$ for some $k \in \NN$, and $\Aut(G) \cong I_h$.
\end{proof}

A set of vertices in a graph is \emph{independent} if there is no edge between
any of its vertices, and the maximum size of an independent set in $G$ is the
\emph{independence number} $\alpha(G)$. Heckman and Thomas~\cite{HecTho06} showed
that every triangle-free, cubic, planar graph has an independent set of size at
least $3n/8$, and this bound is tight. Again, forbidding faces of size greater
than $6$ gives a much better bound.

\begin{corollary}
\label{cor:independence}
  If $G$ is a $3$-connected cubic plane graph on $n$ vertices with all faces of size at
  most $6$, with $p$ pentagonal and $t$ triangular faces, then
  \[
    \alpha(G) \geq n/2-\sqrt{(p+3t)n/20}.
  \]
  In particular, $\alpha(G) \geq n/2-\sqrt{3n/5}$, with equality
  if and only if all faces have size $5$ and $6$, $n=60k^2$ for some
  $k \in \NN$, and $\Aut(G) \cong I_h$.
\end{corollary}

\begin{proof}[Proof of Corollary~\ref{cor:independence}]
  Every graph $G$ contains an odd cycle vertex transversal $U$ such that
  $|U| \leq \tau_{\odd}(G)$, so
  $\alpha(G) \geq \alpha(G\setminus U) \geq n/2-\tau_{\odd}(G)/2$. Therefore, by
  Theorem~\ref{thm:OCT}, $\alpha(G) \geq n/2-\sqrt{3n/5}$, for every $3$-connected
  cubic graph $G$ with all faces of size at most $6$. When $J^*$ is
  a minimum $T$-join of $G^*$, every face of $G^*$ is incident to at most one edge
  of $J^*$. This means that the set $J \subset E(G)$ corresponding to $J^*$ is a
  matching of $G$. Therefore, by Theorem~\ref{thm:OCT}, equality holds if and only
  if all faces have size $5$ and $6$, $n=60k^2$ for some $k \in \NN$, and
  $\Aut(G) \cong I_h$.
\end{proof}

\section{Concluding remarks}\label{sec:conclusion}

Clearly, a necessary condition for $\tau_{\odd}=O(\sqrt n)$ is that
$\nu_{\odd}=O(\sqrt n)$. In the case of planar graphs, the theorem of
Kr\'al\!'\ and Voss~\cite{KraVos04} mentioned
in the introduction guarantees that it is also a sufficient condition.
It can be shown that $\nu_{\odd}=O(\sqrt n)$ is also a
necessary and sufficient condition for having a max-cut of size at least
$3n/2-O(\sqrt n)$, and for having an independent set of size at least $n/2-O(\sqrt n)$.

It is not hard to construct an infinite
family of $3$-connected cubic plane graphs with all faces of size at most $7$
such that $\tau_{\odd} \geq \varepsilon n$, for a constant $\varepsilon > 0$.
This shows that the condition on the size of faces in Theorem~\ref{thm:OCT} and
Corollaries~\ref{cor:max-cut} and~\ref{cor:independence} cannot be relaxed.

To construct such a family, consider the graphs $C$ and $R$ in
Figure~\ref{fig:57graph2}. Note that $C$ is embedded in a disc, and $R$ is embedded
in a cylinder. There are ten vertices on the boundary of $C$ and also on each boundary
of $R$, with the degree alternating between $2$ and $3$. We can paste $k$ copies of
$R$ along their boundaries, and then paste a copy of $C$ on each boundary of the
resulting cylinder. Assuming $k>0$, this gives a $3$-connected cubic plane graph $G$ on
$n=15+40k$ vertices with all faces of size $5$ and $7$, such that
$\nu_{\odd}(G) \geq 4+5k>\tfrac18n$.

\begin{figure}
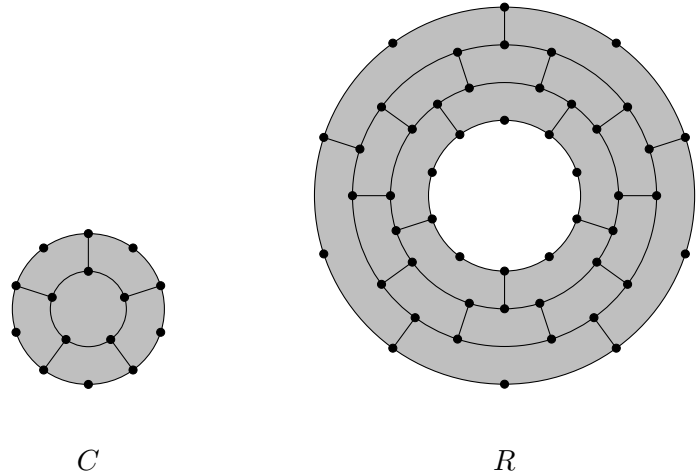

  \centering
  \begin{tikzgraph}[scale=0.5]
    \draw[fill=gray!50] (0,0) circle (2cm);
    \draw (0,0) circle (1cm);
    \foreach \i in {0,...,4}
      \draw (90+72*\i:1)--(90+72*\i:2);
    \foreach \i in {0,...,4}
    {
      \draw (90+72*\i:1) node[vertex] {};
      \draw (90+72*\i:2) node[vertex] {};
      \draw (54+72*\i:2) node[vertex] {};
    }
    \node at (0,-4) {$C$};
  \end{tikzgraph}
  \hfil
  \begin{tikzgraph}[scale=0.5]
    \draw[fill=gray!50] (0,0) circle (5cm);
    \draw (0,0) circle (4cm);
    \draw (0,0) circle (3cm);
    \draw[fill=white] (0,0) circle (2cm);
    \foreach \i in {0,...,4}
      \draw (54+72*\i:2)--(54+72*\i:3);
    \foreach \i in {0,...,9}
      \draw (72+36*\i:3)--(72+36*\i:4);
    \foreach \i in {0,...,4}
      \draw (90+72*\i:4)--(90+72*\i:5);
    \foreach \i in {0,...,9}
    {
      \draw (90+36*\i:2) node[vertex] {};
      \draw (72+36*\i:3) node[vertex] {};
      \draw (72+36*\i:4) node[vertex] {};
      \draw (90+36*\i:5) node[vertex] {};
    }
    \foreach \i in {0,...,4}
    {
      \draw (54+72*\i:3) node[vertex] {};
      \draw (90+72*\i:4) node[vertex] {};
    }
    \node at (0,-7) {$R$};
  \end{tikzgraph}
  \caption{The graphs $C$ and $R$ used to construct an infinite family of
  $3$-connected cubic plane graphs with all faces of size $5$ and $7$ such
  that $\nu_{\odd}>\tfrac18n$.}
\label{fig:57graph2}
\end{figure}

Finally, we remark that bounding $\tau_{\odd}$ is much simpler if the graph contains no
pentagonal faces. In this case, the bound in Theorem~\ref{thm:OCT} can be improved to
$\tau_{\odd}(G) \leq \sqrt{tn/3}$, where $t$ is the number of triangular
faces. In particular, $\tau_{\odd}(G) \leq \sqrt{4n/3}$, with equality if and only if
all faces have size $3$ and $6$, $n=12k^2$ for some $k \in \NN$, and
$\Aut(G) \cong T_d$. Corollaries~\ref{cor:max-cut} and~\ref{cor:independence}
can be strengthened in the same way.

\bibliographystyle{plain}
\bibliography{OCT-cubic}

\end{document}